\documentclass{amsart}
\usepackage{mathabx,mathtools,amsrefs,mathrsfs,math,tikz}
\usetikzlibrary{arrows}
\newcommand\K{{\mathbb K}}
\renewcommand\im{{\operatorname{im}}}
\numberwithin{equation}{section}

\begin{document}
\title{Cellular automata, duality and sofic groups}
\author{Laurent Bartholdi}
\date{June 25, 2017}
\address{D\'epartement de Math\'ematiques et Applications, \'Ecole Normale Sup\'erieure, Paris \textit{and} Mathematisches Institut, Georg-August Universit\"at zu G\"ottingen}
\email{laurent.bartholdi@gmail.com}
%\dedicatory{To Tullio G.\ Ceccherini-Silberstein, on the occasion of his fiftieth anniversary}

\thanks{This work is supported by the ``@raction'' grant ANR-14-ACHN-0018-01}
\begin{abstract}
  We produce for arbitrary non-amenable group $G$ and field $\K$ a
  non-pre-injective, surjective linear cellular automaton. This
  answers positively Open Problem (OP-14) in Ceccherini-Silberstein
  and Coornaert's monograph ``Cellular Automata and Groups''.

  We also reprove in a direct manner, for linear cellular automata,
  the result by Capobianco, Kari and Taati that cellular automata over
  sofic groups are injective if and only if they are post-surjective.

  These results come from considerations related to matrices over
  group rings: we prove that a matrix's kernel and the image of its
  adjoint are mutual orthogonals of each other. This gives rise to a
  notion of ``dual cellular automaton'', which is pre-injective if and
  only if the original cellular automaton is surjective, and is
  injective if and only if the original cellular automaton is
  post-surjective.
\end{abstract}
\maketitle

\section{Introduction}
\subsection{Cellular automata}
Let $G$ be a group and let $\K$ be a field. A \emph{linear cellular
  automaton} on $G$ is --- no more, no less --- a square matrix with
entries in the group ring $\K G$.

The interpretation of a linear cellular automaton
$\Theta\in M_n(\K G)$ is as follows. Let $S$ be a finite subset of $G$
such that all entries of $\Theta$ are in the $\K$-span of
$S$. Construct the graph $\mathscr G$ with vertex set $G$, and with an
edge from $g$ to $g s$ for all $g\in G,s\in S$. Put a copy of the
vector space $V\coloneqq\K^n$ at each vertex of $\mathscr G$.
Elements of the vector space $V^G=\{c\colon G\to V\}$ are called
\emph{configurations}. Then $\Theta$, defines a one-step evolution
rule still written $\Theta$ on the space of configurations, in which
each vertex of $\mathscr G$ inherits a value in $V$ depending on the
values at its neighbours: one may write
$\Theta=\sum_{s\in S}\Theta_s s$ for $\K$-matrices $\Theta_s$, and
then every configuration $c\in V^G$ evolves under $\Theta$ to the
configuration taking at every $g\in G$ the value
$\sum_{s\in S}\Theta_s(c(s^{-1}g))$. More concisely, $c$ evolves to
$\Theta\cdot c$. For more information on linear cellular automata, we
defer to~\cite{ceccherini-coornaert:cag}*{Chapter~8}.

Linear cellular automata are natural linear analogues of classical
cellular automata, in which each vertex of $\mathscr G$ takes a value
in a finite set $A$, which evolves according to the values at its
neighbours. The cellular automaton is thus a locally-defined evolution
rule on the compact space $A^G$. In particular, if $\K$ is a finite
field, then every linear cellular automaton is also a classical
cellular automaton.

The converse, however, is far from true: linear cellular automata are
extremely restricted computational models, and there is no clear way
of converting a classical cellular automaton into a linear one. Every
self-map of a finite set $A$ induces a self-map of the
finite-dimensional vector space $V\coloneqq\K A$, so cellular automata
acting on $A^G$ induce linear self-maps on $\K(A^G)$, but this space
is much larger than $V^G$: the former is a completion of the tensor
power $\bigotimes_G V$ (the ``measuring coalgebra''
$\K G\rightharpoonup V$), while the latter is a completion of the
direct sum $\bigoplus_G V$.

\subsection{Sofic groups and surjunctivity}
How are algebraic properties of the group $G$ reflected in the
cellular automata carried by $\mathscr G$?  We single out some
properties of cellular automata which have received particular
attention: let us write $x\sim y$ for $x,y\in A^G$ when
$\{g\in G\mid x(g)\neq y(g)\}$ is finite. A cellular automaton
$\Theta\colon A^G\righttoleftarrow$ is
\begin{description}
\item[injective] if $\Theta(x)=\Theta(y)$ implies $x=y$;
\item[pre-injective] if $\Theta(x)=\Theta(y)\wedge x\sim y$ implies
  $x=y$; otherwise one calls such $x,y$ \emph{Mutually Eraseable
    Patterns};
\item[surjective] if $\Theta(x)=A^G$; one then says that $\Theta$ has
  no \emph{Garden of Eden};
\item[post-surjective] if $y\sim\Theta(x)$ implies
  $\exists z\sim x: \Theta(z)=y$.
\end{description}

Moore and Myhill's celebrated ``Garden of Eden'' theorem asserts that,
if $G=\Z^d$, then cellular automata are pre-injective if and only if
they are surjective~\cites{moore:ca,myhill:ca}. This has been extended
to \emph{amenable} groups $G$ by Ceccherini-Silberstein, Mach\`\i\ and
Scarabotti~\cite{ceccherini-m-s:ca}, and I proved
in~\cites{bartholdi:moore,bartholdi:myhill} that both results may fail
as soon as $G$ is not amenable. We shall not need the precise
definition of amenable groups; suffice it to say that one of the
equivalent definitions states that $G$ contains finite subsets that
are arbitrarily close to invariant under translation, in the sense
that for every finite $S\subseteq G$ and every $\epsilon>0$ there
exists a finite subset $F\subseteq G$ with
$\#(F S\setminus F)<\epsilon\#F$. We recall:
\begin{thm}[\cite{bartholdi:myhill}]\label{thm:myhill}
  For a group $G$, the following are equivalent:
  \begin{enumerate}
  \item $G$ is non-amenable;
  \item for some integer $n$ and every (or equivalently some) field
    $\K$, there is an injective $G$-module map
    $(\K G)^n\to(\K G)^{n-1}$.
  \end{enumerate}
\end{thm}

We shall also not need the precise definition of \emph{sofic groups},
a common generalization of amenable and residually finite groups; we
refer to the original article~\cite{weiss:sofic}. Suffice it to say
that it is at present unknown whether non-sofic groups exist, and that
if $G$ is sofic then it satisfies Gottschalk's ``Surjunctivity
Conjecture'' from~\cite{gottschalk:general}, namely every injective
cellular automaton is surjective~\cite{weiss:sofic}*{\S3}.
Capobianco, Kaari and Taati show
in~\cite{capobianco-kari-taati:postsurjectivity} that, when $G$ is
sofic, every post-surjective cellular automaton is pre-injective. Thus
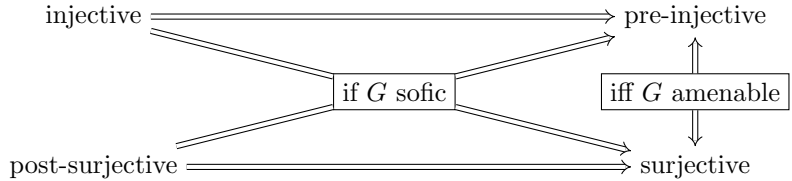
\[\begin{tikzpicture}[¥/.style={double equal sign distance,double,-implies}]
    \node (PS) at (0,0) {post-surjective};
    \node (I) at (0,2) {injective};
    \node (S) at (8,0) {surjective};
    \node (PI) at (8,2) {pre-injective};
    \draw[¥] (I) -- (PI);
    \draw[¥] (PS) -- (S);
    \node[shape=rectangle,draw] (A) at (8,1) {iff $G$ amenable};
    \draw[¥] (A) -- (PI);
    \draw[¥] (A) -- (S);
    \node[shape=rectangle,draw] (X) at (4,1) {if $G$ sofic};
    \draw[double equal sign distance,double] (X) -- (PS);
    \draw[double equal sign distance,double] (X) -- (I);
    \draw[¥] (X) -- (S);
    \draw[¥] (X) -- (PI);
  \end{tikzpicture}
\]

We remark that if a cellular automaton is injective and surjective,
then its inverse is also a cellular automaton. Similarly, if a
cellular automaton is pre-injective and post-surjective, then it is
bijective and its inverse is also a cellular automaton.

The notions of (pre-)injectivity and (post-)surjectivity become
substantially simpler in the context of linear cellular automata, and
exhibit more clearly the duality:
\begin{lem}
  A linear cellular automaton $\Theta\colon V^G\righttoleftarrow$ is
  pre-injective, respectively post-surjective if and only if its
  restriction to $\bigoplus_G V$ is injective, respectively
  surjective.\qed
\end{lem}

Note that non-surjective linear cellular automata
$\Theta\colon V^G\righttoleftarrow$ avoid a non-empty open subset of
$V^G$, namely there exists a finite subset $F\subseteq G$ and
$x\in V^F$ such that $\Theta(y)$ never restricts to $x$ on $F$, see
Proposition~\ref{prop:closed}.

\subsection{A problem by Ceccherini-Silberstein and Coornaert}
Ceccherini-Silberstein and Coornaert prove in~\cite{ceccherini-c:lca}
that if $G$ is an amenable group then a linear cellular automaton on
$G$ is pre-injective if and only if it is surjective, and ask if this
is also a characterization of amenability in the restricted context of
linear cellular automata.

The construction I gave in~\cite{bartholdi:myhill} actually produces,
for every non-amenable group $G$, a pre-injective, non-surjective
linear cellular automaton. Ceccherini-Silberstein and Coornaert ask
in~\cite{ceccherini-coornaert:cag}*{Open Problem 14}:
\begin{problem}\label{problem:main}
  Let $G$ be a non-amenable group and let $\K$ be a field. Does
  there exist a finite-dimensional $\K$-vector space $V$ and a
  linear cellular automaton $\Theta\colon V^G\righttoleftarrow$ which
  is surjective but not pre-injective?
\end{problem}

The group ring $\K G$ admits an anti-involution $*$, defined on basis
elements $g\in G$ by $g^*\coloneqq g^{-1}$ and extended by
linearity. It induces an anti-involution on $M_n(\K G)$ as follows:
for $\Theta\in M_n(\K G)$, set $(\Theta^*)_{i j}=\Theta_{j i}^*$ for
all $i,j\in\{1,\dots,n\}$; namely, $\Theta^*$ is computed from
$\Theta$ by transposing the matrix and applying $*$ to all its
entries. Clearly $\Theta^{**}=\Theta$. There is a natural bilinear
pairing $(\K G)^n\times(\K^n)^G\to\K$, given by
\[\langle v|\xi\rangle\coloneqq\sum_{g\in G}\sum_v(g)\cdot\xi(g).\]
In this article, I shall prove:
\begin{thm}\label{thm:main}
  Let $G$ be a group, let $\K$ be a field, and let
  $\Theta\in M_n(\K G)$ be a linear cellular automaton. Then
  \begin{align}
    \ker(\Theta|\K^n G)^\perp &=\im(\Theta^*|(\K^n)^G),\label{eq:main:1}\\
    \ker(\Theta|(\K^n)^G)^\perp &=\im(\Theta^*|\K^n G),\label{eq:main:2}\\
    \im(\Theta|\K^n G)^\perp &= \ker(\Theta^*|(\K^n)^G),\label{eq:main:3}\\
    \im(\Theta|(\K^n)^G)^\perp &= \ker(\Theta^*|\K^n G).\label{eq:main:4}
  \end{align}
  In particular, $\Theta$ is pre-injective if and only if $\Theta^*$
  is surjective, and $\Theta$ is injective if and only if $\Theta^*$
  is post-surjective.
\end{thm}

\noindent This answers positively Problem~\ref{problem:main}:
\begin{cor}\label{cor:main}
  Let $G$ be a non-amenable group and let $\K$ be an arbitrary
  (possibly finite) field. Then there exist surjective,
  non-pre-injective linear cellular automata on $G$.
\end{cor}
\begin{proof}
  Let $\Theta\in M_n(\K G)$ be a pre-injective, non-surjective linear
  cellular automaton, obtained e.g.\ by adding a full row of $0$'s to
  the matrix given by Theorem~\ref{thm:myhill}. Then $\Theta^*$ is the
  required example.
\end{proof}

\subsection{Capobianco, Kari and Taati's result}
From this duality of linear cellular automata, one also deduces an
immediate proof of Capobianco, Kari and Taati's main result, when
restricted to linear cellular automata:
\begin{thm}[see~\cite{capobianco-kari-taati:postsurjectivity}*{Theorem~2}]
  Let $G$ be a sofic group. Then every post-surjective linear cellular
  automaton is pre-injective.
\end{thm}
\begin{proof}
  Let $\Theta$ be a post-surjective linear cellular automaton. By
  Theorem~\ref{thm:main}, $\Theta^*$ is injective, so $\Theta^*$ is
  surjective by~\cite{weiss:sofic}*{\S3}, so $\Theta$ is pre-injective
  again by Theorem~\ref{thm:main}.
\end{proof}

\subsection{Reddite ergo qu\ae\ C\ae saris sunt}
The notion of dual linear cellular automata is quite natural, but its
first appearance seems only to be a passing remark
in~\cite{popovici2:dilatability}. The last line of
Theorem~\ref{thm:main} has been proven, in the setting of locally
finite graphs, by Matthew Tointon in~\cite{tointon:harmonic}. I am
indebted to Professor Coornaert for having pointed out that reference
to me when I shared this note with him.

In a recent article~\cite{gaboriau-seward:soficentropy}, Gaboriau and
Seward study the sofic entropy of algebraic actions, and note the
following consequence of Corollary~\ref{cor:main}: if $G$ is sofic but
not amenable, then the Yuzvinsky addition formula for entropy
$h(G\looparrowright A)=h(G\looparrowright B)+h(G\looparrowright A/B)$
fails for some $G$-modules $B\le A$. Indeed take $A=(\K^n)^G$ and
$B=\ker(\Theta)$ for a surjective, non-pre-injective cellular
automaton $\Theta$. I am grateful to Messrs. Gaboriau and Seward for
having communicated their remark to me ahead of its publication.

\section{Linear cellular automata}\label{ss:lca}
We start with a field $\K$ and an integer $n$. We write
$V\coloneqq\K^n$, and identify $V$ with $V^*$.  There is a natural
bilinear, non-degenerate pairing $V^*\times V\to\K$ given by
\[\langle\phi|v\rangle=\phi(v)=\sum_{i=1}^n \phi_i v_i.\]

Let $G$ be a group. We denote by $V^G$ the vector space of functions
$G\to V$, and declare its closed subsets to be
$\{c\in V^G\mid c_{|S}\in W\}$ for all finite $S\subseteq G$ and all
$W\le V^S$. In particular, the restriction maps
$\pi_S\colon V^G\to V^S$ are continuous for all finite $S\subseteq G$,
and $V^G$ is compact (but not Hausdorff).

We denote by $V^*G$ the vector subspace of finitely-supported
functions in $V^G$. There is a left action of $G$ on $V^G$ by
translation: for $g\in G,c\in V^G$ we define $g c\in V^G$ by
$(g c)(h)=c(g^{-1} h)$. This action preserves $V^*G$. There is also a
bilinear pairing
\[\langle{-}|{-}\rangle\colon V^*G\times V^G\to\K,\qquad\langle\omega|c\rangle=\sum_{g\in G}\langle\omega(g)|c(g)\rangle.\]
\begin{lem}
  $\langle{-}|{-}\rangle$ is non-degenerate in both arguments.\qed
\end{lem}

In the notation introduced above, a linear cellular automaton is both
an element of $V\otimes V^*G$ and a $G$-equivariant, continuous
self-map $\Theta\colon V^G\righttoleftarrow$. Note that $\Theta$
restricts to a self-map $V^*G\righttoleftarrow$.
\begin{prop}\label{prop:closed}
  Let $\Theta\colon V^G\righttoleftarrow$ be a linear cellular
  automaton. Then $\Theta(V^G)$ is a closed subspace of $V^G$.
\end{prop}
\begin{proof}
  Verbatim the proof
  of~\cite{ceccherini-coornaert:cag}*{Theorem~8.8.1}.  Note that they
  claim in fact the weaker statement that $\Theta(V^G)$ is closed in
  the prodiscrete topology. Note also that the proposition does not
  follow trivially from the fact that $V^G$ is compact, because $V^G$
  is not Hausdorff.
\end{proof}

Consider a linear cellular automaton $\Theta\in V\otimes V^*G$,
written as $\Theta=\sum_i v_i\otimes\phi_i g_i$ for finitely many
$v_i\in V,\phi_i\in V^*,g_i\in G$. Then, tracing back to our original
definition, its adjoint $\Theta^*\in V^*\otimes V G$ is
$\Theta^*=\sum_i\phi_i\otimes v_i g_i^{-1}$.
\begin{lem}
  Let $\Theta\in V\otimes V^*G$ be a cellular automaton, with adjoint
  $\Theta^*$. Then
  \[\langle\Theta^*(\omega)|c\rangle=\langle\omega|\Theta(c)\rangle\text{ for all }\omega\in V^*G,c\in V^G.\]
\end{lem}
\begin{proof}
  Write $\Theta$ as a finite sum $\sum_i v_i\otimes\phi_i g_i$. Then
  the sides of the above equation are respectively
  \begin{align*}
    \sum_{g\in G}\Big\langle\Big\{\sum_i\phi_i\otimes v_i (g_i^{-1}\omega)\Big\}(g)\Big|c(g)\Big\rangle
    = \sum_{g\in G,i}\langle\phi_i|c(g)\rangle\;\langle\omega(g_i g)|v_i\rangle
    \intertext{and}
    \sum_{g\in G}\Big\langle\omega(g)\Big|\Big\{\sum_i v_i\otimes\phi_i(g_i c)\Big\}(g)\Big\rangle
    = \sum_{g\in G,i}\langle\omega(g)|v_i\rangle\;\langle\phi_i|c(g_i^{-1} g)\rangle,
  \end{align*}
  which are just permutations of each other.
\end{proof}

\section{Proof of Theorem~\ref{thm:main}}
Let $\Theta\in M_n(\K G)$ be a linear cellular automaton, and as
in~\S\ref{ss:lca} set $V=V^*=\K^n$, with the usual scalar product.

We begin by the inclusion
$\ker(\Theta|V^*G)^\perp\supseteq\im(\Theta^*|V^G)$
from~\eqref{eq:main:1}. Given $c\in\im(\Theta^*|V^G)$, say
$c=\Theta^*(d)$, for all $\omega\in\ker(\Theta|V^*G)$ we have
\[\langle\omega|c\rangle=\langle\omega|\Theta^*(d)\rangle=\langle\Theta(\omega)|d\rangle=\langle0|d\rangle=0,\]
so $c\perp\ker(\Theta|V^*G)$. The exact same computation gives all
`$\supseteq$' inclusions from~\eqref{eq:main:2}, \eqref{eq:main:3}
and~\eqref{eq:main:4}.

We continue with the inclusion
$\ker(\Theta|V^*G)\perp\subseteq\im(\Theta^*|V^G)$
from~\eqref{eq:main:1}. Given $c\not\in\im(\Theta^*|V^G)$, there
exists an open neighbourhood of $c$ in $V^G\setminus\im(\Theta^*|V^G)$
by Proposition~\ref{prop:closed}; so there exists a finite subset
$S\subseteq G$ and a proper subspace $W<V^S$ such that the projection
$\pi_S(V^G)$ belongs to $W$. Since $V^S$ is finite-dimensional, there
exists a linear form $\omega$ on $V^S$ that vanishes on $W$ but does
not vanish on $c$. Note that $\omega$, qua element of $(V^S)^*$, is
canonically identified with an element of $(V^*)^S$, and therefore
with an element of $V^*G$. From $\omega\perp\im(\Theta^*|V^G)$ we get
$\Theta(\omega)\perp V^G$ so $\Theta(\omega)=0$ because the scalar
product $\langle{-}|{-}\rangle$ is non-degenerate. Therefore
$c\not\perp\ker(\Theta|V^*G)$ as desired.

We continue with the inclusion
$\ker(\Theta|V^G)\perp\subseteq\im(\Theta^*|V^*G)$
from~\eqref{eq:main:2}. Given $\omega\not\in\im(\Theta^*|V^*G)$, there
exists a linear form $c\in(V^*G)^*$ that vanishes on
$\im(\Theta^*|V^*G)$ but does not vanish on $\omega$. Note that
$(V^*G)^*$ canonically identifies with $V^G$. From
$c\perp\im(\Theta^*|V^*G)$ we get $\Theta(c)\perp V^*G$, so
$\Theta(c)=0$ because the scalar product $\langle{-}|{-}\rangle$ is
non-degenerate. Therefore $\omega\not\perp\ker(\Theta|V^G)$ as desired.

We finally consider the inclusion
$\im(\Theta|V^*G)^\perp\subseteq\ker(\Theta^*|V^G)$
from~\eqref{eq:main:3}. Given $c\perp\im(\Theta|V^*G)$, we have
$c\perp\Theta(\omega)$ for all $\omega\in V^*G$, so
$\Theta^*(c)\perp\omega$ for all $\omega\in V^*G$, so
$\Theta^*(c)\perp V^*G$ and therefore $\Theta^*(c)=0$ because the
scalar product $\langle{-}|{-}\rangle$ is non-degenerate. The exact
same computation gives the `$\subseteq$' inclusion
from~\eqref{eq:main:4}.

Recalling that $\Theta$ is pre-injective if and only if
$\ker(\Theta|V^*G)=0$ and $\Theta$ is injective if and only if
$\ker(\Theta|V^G)=0$ and $\Theta$ is post-surjective if and only if
$\im(\Theta|V^*G)=V^*G$ and $\Theta$ is surjective if and only if
$\im(\Theta|V^G)=V^G$, the last conclusions follow.

\end{document}